\documentclass[a5paper, 10pt]{article}

\usepackage{cite, amsmath, amssymb, booktabs, lastpage, graphicx}
\usepackage[hidelinks]{hyperref}

\usepackage{geometry}
\usepackage{setspace}

\usepackage{amsthm}

\theoremstyle{plain}
\newtheorem{lemma}{Lemma}
\newtheorem{theorem}{Theorem}
\newtheorem{conjecture}{Conjecture}

\newtheorem{claim}{Claim}
\newtheorem{observation}{Observation}

\theoremstyle{remark}

\theoremstyle{definition}

\usepackage{graphicx}
\makeatletter
\renewcommand{\maketitle}{
	\begin{center}

		\baselineskip=0.30in
		{\Large\bfseries \@title} \par
		\vspace{5mm}
		\baselineskip=0.2in
		{\large\bfseries \@author}\par
		\vspace{1mm}
		{\it \@address} \par
		{\small\tt \@email} \par
		\vspace{3mm}
		
	\end{center}
	\vspace{3mm}
}

\newcommand{\address}[1]{\def\@address{#1}}
\newcommand{\email}[1]{\def\@email{#1}}

\makeatother

\newcommand{\acknowledgment}[1]{\vspace{5mm}\singlespacing
	{\noindent\textbf{Acknowledgment\/} #1}
}

\usepackage{fancyhdr}

\usepackage[margin=1cm,%
			font=footnotesize,%
			format=hang,%
			labelsep=period,%
			labelfont=bf]{caption}

\title{Packing spanning arborescences with extra large one}
\author{Hui Gao}

\address{School of Mathematics, China University of Mining and Technology, Xuzhou, Jiangsu 221116, China }

\email{gaoh1118@yeah.net}

\date{\today}

\begin{document}

\maketitle

\begin{abstract}
The celebrated Nash-Williams and Tutte's theorem states that a graph $G=(V, E)$ contains $k$ edge disjoint spanning
trees if and only if $\nu_{f}(G) \geq k$, where
$$\nu_{f}(G):=\min_{|\mathcal{\mathcal{P}}|>1, \text{$\mathcal{P}$ is a partition of $V(G)$}}\frac{|E( \mathcal{P})|}{|\mathcal{P}|-1}.$$
Inspired by the NDT theorem as structural explanations for the fractional part of Nash-Williams' forest decomposition theorem, Fang and Yang extended Nash-Williams and Tutte's theorem and proved that if $\nu_{f}(G) > k+ \frac{d-1}{d}$, then $G$ contains $k$ edge disjoint spanning trees and another forest $F$ with $ |E(F)|> \frac{d-1}{d} (|V(G)|-1)|$, and if $F$ is not a spanning tree, then $F$ has a component with at least $d$ edges. In this paper, we give a digraphic version of their result; however, the mixed graphic version remains open.
\end{abstract}



\onehalfspacing

\section{Introduction}


All graphs, digraphs and mixed graphs in this paper are considered to be multiple, that is, they can have multiple edges or arcs but not loops.

Throughout the introduction, let $G=(V, E)$ be a graph, $D=(V, A)$ be a digraph, and $X, Y \subseteq V$ ($X$ and $Y$ don't have to be disjoint).
Denote by $A[X,Y]$ the set of arcs in $A$ with their tails in $X$ and heads in $Y$; when $X=Y$, denote $A(X):=A[X, X]$ and $d_{A}^{-}(X):=|A(V \setminus X, X)|$. Denote by $E(X)$ the set of edges in $E$ with both end vertices in $X$.
Sometimes, if $D_{0}$ is a subdigraph of $D$, we write $d_{D_{0}}^{-}(X):=d_{A(D_{0})}^{-}(X)$.
Let $X_{1}, \ldots, X_{t}$ be disjoint subsets of $V$, we call $\mathcal{P}= \{X_{1}, \ldots, X_{t} \}$ a {\em subpartition of $V$} and particularly a {\em partition of  $V$} if $V=\cup_{j=1}^{t}X_{j}$. Denote by $\mathcal{D}(V)$ the set of all subpartitions of $V$. Denote $\cup \mathcal{P}:=\cup_{X \in \mathcal{P}}X$ for $\mathcal{P} \in \mathcal{D}(V)$.

To {\em shrink} $X$ in $D$ is to delete all  arcs with both end vertices in $X$ and then merge the vertices of $X$ into a single vertex. We denote the resulting digraph  by $D/X$. For a subpartition $\mathcal{P}$ of $V$, denote by $E(\mathcal{P})$ the set  consisting of  all edges in $E$ with end vertices belonging to different subsets in $\mathcal{P} \cup \{V \setminus \cup \mathcal{P} \}$. Let $A_{0}, A_{1}$ be two sets of arcs with both end vertices in $V$ such that $A_{0} \subseteq A$ and $A_{1} \cap A = \emptyset$. Denote by $D-A_{0}$ the subdigraph $(V, A\setminus A_{0})$, by $D+A_{1}$ the digraph $(V, A \cup A_{1})$, and by $D[X]$ the subdigraph $(X, A(X))$, respectively. A subdigraph $D_{0}$ is said to {\em span $X$} if $V(D_{0})=X$ and is said to be {\em spanning} if $V(D_{0})=V$.

A digraph $T$ is called an \emph{$r$-arborescence} if its underlying graph is a tree, and for any $v \in V(T)$,  there is exactly one directed path in $T$ from $r$ to $v$.  We say the vertex $r$ is the {\em root of $T$} or $T$ is {\em rooted at $r$}. A digraph $F$ is called a \emph{branching} if each component of $F$ is an arborescence. The \emph{root set} $R(F)$ of $F$  consists of all roots of its components; We call $F$ a \emph{$c$-branching} if $|R(F)|=c$.

For simplicity,  we write $[k]:=\{1, 2, \ldots, k \}$ for $k \in \mathbb{Z}^{+}$, $X+x_{0}:=X \cup \{x_{0} \}$ for $x_{0} \notin X$, and $X-x_{1}:=X \setminus \{x_{1} \}$ for $x_{1} \in X$.
For all notations and terminology used, but not defined here, we refer to the textbook~\cite{BM-08}.

The celebrated Nash-Williams' forest decomposition theorem (Theorem~\ref{Nash-forest-decomposition}) characterized a graph which can be decomposed into $k$ edge disjoint forest. As its dual, Nash-Williams and Tutte's theorem (Theorem~\ref{Nash-Williams-Tutte})  characterized  a graph $G$ having $k$ edge disjoint spanning trees.
The {\em fractional arboricity} $\gamma_{f}(G)$ is defined by
$$\gamma_{f}(G):= \max_{X \subseteq V(G), |X|>1}\frac{|E(X)|}{|X|-1}.$$
The {\em fractional packing number} $\nu_{f}(G)$ is defined by
$$\nu_{f}(G):=\min_{|\mathcal{\mathcal{P}}|>1, \text{$\mathcal{P}$ is a partition of $V(G)$}}\frac{|E( \mathcal{P})|}{|\mathcal{P}|-1}.$$

\begin{theorem}[\cite{N-64}]\label{Nash-forest-decomposition}
A graph $G$ can be decomposed into $k$ edge disjoint forests if and only if $\gamma_{f}(G) \leq k$.
\end{theorem}

\begin{theorem}[\cite{N-61,T-61}]\label{Nash-Williams-Tutte}
A graph $G$ contains $k$ edge disjoint spanning trees if and only if $\nu_{f}(G) \geq k$.
\end{theorem}

The Nine Dragon Tree (NDT) Theorem (Theorem~\ref{NDT}), conjectured by Montassier et. al. \cite{MMRZ-12} and proved by Jiang and Yang \cite{JY-17},  acts as a structural explanation for the fractional part of the fractional arboricity. And later, Gao and Yang \cite{GY-23} gave a digraphic version of the NDT Conjecture and prove part of it. A graph $G$ is called $d$-bounded if the maximum degree of $G$ is at most $d$.

\begin{conjecture}[NDT Conjecture, \cite{MMRZ-12}]\label{conjecture-NDT}
Let $G$ be a graph, $k, d$ be nonnegative integers. If $\gamma_{f}(G) \leq k+\frac{d}{d+k+1} $, then $G$ decomposes into $k + 1$ forests with one being $d$-bounded.
\end{conjecture}

\begin{theorem}[NDT Theorem, \cite{JY-17}]\label{NDT}
The Nine Dragon Tree Conjecture is true.
\end{theorem}

\begin{conjecture}[\cite{GY-23}]\label{digraph-NDT}
Let $D$ be a digraph, $k, d$ be positive integers. If  $ \gamma(D) \leq k+ \frac{d-k}{d+1}$  and $\Delta^{-}(D) \leq k+1$, then $D$ decomposes into $k + 1$ branchings $F_{1}, F_{2}, \ldots, F_{k+1}$  with $\Delta^{+}(F_{k+1}) \leq d+1 $.
\end{conjecture}

\begin{theorem}[\cite{GY-23}]\label{part-digraph-NDT}
Conjecture~\ref{digraph-NDT} is true for $d \leq k$.
\end{theorem}

Inspired by the NDT Theorem, Fang and Yang \cite{FY-24} showed Theorem~\ref{undirected-version}, acting as a structural explanation for the fractional part of the fractional packing number.

\begin{theorem}[\cite{FY-24}]\label{undirected-version}
Let $G$ be a graph, $k \geq 0$, $d \geq 1$ be integers.  If $\nu_{f}(G) > k+ \frac{d-1}{d}$, then $G$ contains edge disjoint $k$  spanning trees and another forest $F$ with
$|E(F)|> \frac{d-1}{d}(|V(G)|-1)$, and if $F$ is not a spanning tree, then $F$ has a component with at least $d$ edges. Moreover, the bound of $\nu_{f}(G)  $ is sharp.
\end{theorem}

In this paper, we give a digraphic version of Theorem~\ref{undirected-version}, also acting as a fractional  version of Theorem~\ref{packing-spanning-arborescence}.
The {\em fractional packing number} $\nu_{f}(D)$ of a digraph $D$ is defined by
$$\nu_{f}(D):=\min_{|\mathcal{P}|>1, \mathcal{P} \in \mathcal{D}(V)} \frac{\sum_{X \in \mathcal{P}}d_{D}^{-}(X)}{|\mathcal{P}|-1}.$$

\begin{theorem}\label{directed-version}
Let $D$ be a digraph, $k \geq 0$, $d \geq 1$ be integers.  If $\nu_{f}(D) > k+ \frac{d-1}{d}$, then $D$ contains arc disjoint $k$  spanning arborescences and another branching $F$ with
$|A(F)|> \frac{d-1}{d}(|V(D)|-1)$ , and if $F$ is not a spanning arborescence, then $F$ has a component with at least $d$ arcs. Moreover, the bound of $\nu_{f}(D)  $ is sharp.
\end{theorem}

\begin{theorem}[\cite{F-09}]\label{packing-spanning-arborescence}
A digraph $D$ contains $k$ arc disjoint spanning arborescences if and only if $\nu_{f}(D) \geq k$.
\end{theorem}

The paper is organized as follows: we introduce some lemmas and the so called ``properly intersecting elimination operation'' in Section 2, prove Theorem~\ref{directed-version} in Section 3 and  conclude with an open conjecture for the mixed graphic generalization of Theorem~\ref{undirected-version} and \ref{directed-version} in Section 4.

\section{Preliminaries}

\begin{lemma}[\cite{GY-21}]\label{pack-branching}
Let $D=(V, A)$ be a digraph. $k \geq 0$ be an integer, $c_{1}, c_{2}, \ldots, c_{k+1} \leq |V|$ be integers and $U_{1}, U_{2}, \ldots, U_{k+1} \subseteq V$ be such that $|U_{i}| \leq c_{i}$ for $1 \leq i \leq k+1$. Then there exist arc disjoint $k+1$ spanning branchings $F_{1}, F_{2}, \ldots, F_{k+1}$ such that $R(F_{i}) \supseteq U_{i}$ and $|R(F_{i})|=c_{i}$ for $1 \leq i \leq k+1$, if and only if, for any $I \subseteq [k+1]$ and any $\mathcal{P} \in \mathcal{D}(V)$,
\begin{equation}\label{condition-pack-branching}
\sum_{X \in \mathcal{P}}d_{A}^{-}(X) \geq \sum_{X \in \mathcal{P}}|P_{I}(X)|-\sum_{i \in I}(c_{i}-|U_{i}|),
\end{equation}
where $P_{I}(X):= \{ i \in I: X \cap U_{i}=\emptyset\}$.
\end{lemma}

In Lemma~\ref{pack-branching}, by setting $c_{1}= \ldots =c_{k}=1$, $c_{k+1}=c$, $U_{1}= \ldots = U_{k}= \emptyset$ and $U_{k+1}=U$, we have the following lemma.

\begin{lemma}\label{k+extra}
Let $D=(V, A)$ be a digraph. $k \geq 1$ be an integer, $c \leq |V|$ be an integer and $U \subseteq V$ be such that $|U| \leq c$. Then there exist arc disjoint $k$ spanning arborescences and another spanning branching $F$ with $R(F) \supseteq U$ and $|R(F)|=c$, if and only if, for any  $\mathcal{P} \in \mathcal{D}(V)$,
\begin{equation}\label{pre-arborescence-part}
\sum_{X \in \mathcal{P}}d_{A}^{-}(X) \geq k(|\mathcal{P}|-1),
\end{equation}
\begin{equation}\label{pre-branching-part}
\sum_{X \in \mathcal{P}}d_{A}^{-}(X) \geq k(|\mathcal{P}|-1)+ \sum_{X \in \mathcal{P}}|P_{k+1}(X)|-(c-|U|),
\end{equation}
where $P_{k+1}(X):= \{ k+1: X \cap U = \emptyset\}$.
\end{lemma}

Let $V$ be a finite set.
Two subsets $X, Y \subseteq V$ are {\em intersecting} if $X \cap Y \neq \emptyset$ and {\em properly intersecting} if $X \cap Y$, $X \setminus Y$, and $Y \setminus X \neq \emptyset$.
A function $p: 2^{V} \rightarrow \mathbb{Z}$ is  {\em supermodular}, where $2^{V}$ denotes the power set of $V$, if the inequality
\[
p(X)+p(Y) \leq p(X \cup Y)+ p(X \cap Y)
\]
holds for all subsets of $V$. A function $b$ is  {\em submodular} if $-b$ is supermodular.

Let $\mathcal{F}$ be a multiset, which consists of some subsets of $V$ (these subsets do not have to be different). Let $\cup \mathcal{F}$ be the union of elements in $\mathcal{F}$ (then $\cup \mathcal{F} \subseteq V$).
If there exist no properly intersecting pairs in $\mathcal{F}$, then $\mathcal{F}$ is called  {\em laminar}. If there exists a properly intersecting pair $X$ and $Y$ in $\mathcal{F}$, then we obtain $\mathcal{F}'$ from $\mathcal{F}$ by replacing $X$ and $Y $ with $X \cup Y$ and $X \cap Y$. To correspond what have already been  used in  \cite{GY-21,GY-21-spanning}, we name it as   {\em properly intersecting elimination operation of type $1$}, and denote it by  PIEO$^{1}$.

Let $Z_{1}$ and $Z_{2}$ be multisets. Denote by $Z_{1} \uplus Z_{2}$ the {\em multiset union} of $Z_{1}$ and $Z_{2}$, that is, for any $z$, the number of $z$ in $Z_{1} \uplus Z_{2}$ is the total number of $z$ in $Z_{1}$ and $Z_{2}$.

From now on, we suppose $\mathcal{F}_{1}, \mathcal{F}_{2} \in \mathcal{D}(V)$.
We adopt PIEO$^{1}$s in $\mathcal{G}_{0}=\mathcal{F}_{1} \uplus \mathcal{F}_{2}$, step by step, and obtain families   $\mathcal{G}_{0}, \ldots, \mathcal{G}_{i-1},\mathcal{G}_{i},\ldots$ of subsets  of $V$.
As noted in \cite{GY-21,GY-21-spanning}, the process will terminate and suppose the final obtained family is $\mathcal{G}_{n}$.
Let $\mathcal{G}_{i}'$ be the family of maximal elements in $\mathcal{G}_{i}$, and we want all elements of $\mathcal{G}_{i}'$ are distinct from each other (to realize this,  if $X$ is 
a maximal element in $\mathcal{G}_{i}$, no matter how many times, we add exactly one X to  $\mathcal{G}_{i}'$).
Define $\mathcal{F}_{3}:= \mathcal{G}_{n}' $ and $\mathcal{F}_{4}:=\mathcal{G}_{n} \setminus \mathcal{F}_{3} $.
Also as noted in \cite{GY-21,GY-21-spanning}, 	$\mathcal{F}_{3}, \mathcal{F}_{4} \in \mathcal{D}(V)$ and $ |\mathcal{F}_{1}|+ |\mathcal{F}_{2}|= |\mathcal{F}_{3}|+ |\mathcal{F}_{4}|$.

\section{Proof of Theorem~\ref{directed-version}}

We prove Theorem~\ref{directed-version} by induction on $|V(D)|$.

Since $\nu_{f}(D)>k+\frac{d-1}{d}$, we have for any $\mathcal{P} \in \mathcal{D}(V)$, $\frac{\sum_{X \in \mathcal{P}} d_{A}^{-}(X)}{|\mathcal{P}|-1}>k+\frac{d-1}{d}$. Thus for any $\mathcal{P} \in \mathcal{D}(V)$, we have
\begin{equation}\label{integer-part}
\sum_{X \in \mathcal{P}} d_{A}^{-}(X)>k(|\mathcal{P}|-1),
\end{equation}
\begin{equation}\label{frac-part-mid}
\begin{split}
\sum_{X \in \mathcal{P}} d_{A}^{-}(X)
& >(k+\frac{d-1}{d})(|\mathcal{P}|-1)\\
& =k(|\mathcal{P}|-1)+ |\mathcal{P}|-(1+\frac{1}{d}(|\mathcal{P}|-1))\\
& \geq k(|\mathcal{P}|-1)+ |\mathcal{P}|-(1+\frac{1}{d}(|V|-1))~\text{(since $|\mathcal{P}|\leq |V|$)}.
\end{split}
\end{equation}
Set $c:= \lceil \frac{|V|-1}{d} \rceil$. Then since $\sum_{X \in \mathcal{P}} d_{A}^{-}(X) $ is integral, by (\ref{frac-part-mid}),   we have for any $\mathcal{P} \in \mathcal{D}(V)$,
\begin{equation}\label{frac-part}
\sum_{X \in \mathcal{P}} d_{A}^{-}(X) \geq k(|\mathcal{P}|-1)+ |\mathcal{P}|-c.
\end{equation}
By Lemma~\ref{k+extra}, setting $U:= \emptyset$, it follows from (\ref{integer-part}) and (\ref{frac-part}) that there exist arc disjoint $k$ spanning arborescences $T_{1}, \ldots, T_{k}$ and another spanning  $c$-branching $F$ in $D$. If $c=1$, then $F$ is  spanning, completing the proof. Next, suppose $c \geq 2$.

\textbf{Case 1}: The equality of (\ref{frac-part}) doesn't hold, that is, for any $\mathcal{P} \in \mathcal{D}(V)$, we have
\begin{equation}\label{c-1}
\sum_{X \in \mathcal{P}} d_{A}^{-}(X) \geq k(|\mathcal{P}|-1)+ |\mathcal{P}|-(c-1).
\end{equation}
By Lemma~\ref{k+extra}, setting $U:= \emptyset$, it follows from (\ref{integer-part}) and (\ref{c-1}) that there exist $k$ spanning arborescences and another spanning  $(c-1)$-branching $F_{1}$ in $D$. Then $|A(F_{1})|=|V|-(c-1)>\frac{d-1}{d} (|V|-1)$. If $c=2$, then $F_{1}$ is  spanning; if $c \geq 3$, since $\frac{|V|}{c-1} >\frac{|V|}{ \frac{|V|-1}{d}}>d$, $F_{1}$ has a component with at least $d+1$ vertices and thus at least $d$ arcs, completing the proof.

\textbf{Case 2}: There exists some $\mathcal{P}_{0} \in \mathcal{D}(V)$ that witnesses the equality of (\ref{frac-part}).

\begin{claim}\label{contraction}
Let $  W \subsetneq V$ be such that
\begin{itemize}
\item[(i)] $|W| \geq 2$,
\item[(ii)] $A^{-}(W)$ consists of $\overrightarrow{x_{1}y_{1}}, \ldots, \overrightarrow{x_{s}y_{s}}$,
\item[(iii)] there exist $k+1$ arc disjoint spanning arborescences in $D[W]$, $s$ of which are rooted at $y_{1}, \ldots, y_{s}$, respectively.
\end{itemize}
Then there exists such a packing as Theorem~\ref{directed-version} demands.
\end{claim}
\begin{proof}

By induction hypothesis, in $D/W$ (denote by $w$ the contraction vertex of $W $, and $A(D/W)^{-}(w)$ consists of $\overrightarrow{x_{1}w}, \ldots, \overrightarrow{x_{s}w}$), there exist $k$   spanning arborescences $T_{1}', \ldots ,T_{k}'$ and another spanning branching $F'$ with $|A(F')|> \frac{d-1}{d}(|V(D/W)|-1)$, and if $F'$ is not  spanning, then $F'$ has a component with at least $d$ arcs.
Suppose $D[W]$ contains $k+1$  spanning arborescences $T_{1}'', \ldots,  T_{k+1}''$ such that $T_{j}''$ is rooted at $y_{j}$ for $1 \leq j \leq s$.

For $1 \leq i \leq k$, if $d_{T_{i}'}^{-}(w)=1$, say $\overrightarrow{x_{j}w} \in A(T_{i}')$, let $T_{i}^{*}:=(V, A(T_{i}')\cup A(T_{j}'')) $; then $T_{i}^{*}$ is a spanning  arborescence in $D$ with the same root as $T_{i}'$; if $d_{F'}^{-}(w)=1$, the same applies to the above discussion and we obtain a spanning branching $F^{*}$ in $D$ with the same root set as $F'$. For the others in $\{T_{1}', \ldots, T_{k}', F'\}$, there are the same number of arborescneces left in $\{T_{1}'', \ldots, T_{k+1}'' \}$. Fix an arbitrary one-to-one correspondence, say $ T_{i}'$ correspond to $T_{j}''$; let $T_{i}^{*}:=(V, A(T_{i}')\cup A(T_{j}'')) $; then $T_{i}^{*}$ is a spanning  arborescence in $D$ with the same root as $T_{j}''$. Finally, we obtain arc disjoint spanning arborescences $T_{1}^{*}, \ldots, T_{k}^{*}$ and spanning branching $F^{*}$.

Note that $|A(F^{*})|=|A(F')|+|W|-1>\frac{d-1}{d}(|V(D/W)|-1)+|W|-1=\frac{d-1}{d}(|V|-|W|)+|W|-1>\frac{d-1}{d}(|V|-1)$; $F^{*}$ is spanning in $D$ if and only if $F'$ is spanning in $D/W$. If $F^{*}$ is not spanning, then since $F'$ has a component with at least $d$ arcs, $F^{*}$ has a component with at least $d$ arcs.
\end{proof}

Suppose $F$ consists of $c$ arborescences $T_{k+1}^{1}, \ldots, T_{k+1}^{c}$ and $|V(T_{k+1}^{c})|$ is minimum among all such  packings $ \{T'''_{1}$, $\ldots$, $T'''_{k}, F''' \}$ in $D$, where $T'''_{1}$, $\ldots$, $T'''_{k}$ are spanning arborescences, $F'''$ is a spanning  $c$-branching and  $T_{k+1}^{c}$ is a component of $F'''$.

If $|V(T_{k+1}^{c})|=1$, then
\[
\frac{|\cup_{i=1}^{c-1}V(T_{k+1}^{i})|}{c-1}=\frac{|V|-1}{c-1}>\frac{|V|-1}{\frac{|V|-1}{d}}=d.
\]
So there exists some $ 1 \leq i_{0} \leq c-1$ such that $|V(T_{k+1}^{i_{0}})| \geq d+1$, completing the proof. Next, suppose $|V(T_{k+1}^{c})| \geq 2$.
Suppose $r \in V(T_{k+1}^{c})$ is the root of $T_{k+1}^{c}$, and $r_{i} \in V$ is the root of $T_{i}$ for $1 \leq i \leq k$.

\textbf{Case 2.1} $A[\cup_{i=1}^{c-1}V(T_{k+1}^{i}), V(T_{k+1}^{c})]=\emptyset$.

For $ 1 \leq i \leq k$, since $T_{i}$ spans $V$, $r_{i}$ can reach each vertex of $V$ in $T_{i}$.
Since $A[\cup_{i=1}^{c-1}V(T_{k+1}^{i}), V(T_{k+1}^{c})]=\emptyset $,  we have  $r_{i} \in V(T_{k+1}^{c})$ and $T_{i}[V(T_{k+1}^{c}]$ is an arborescence. So there exists $k+1$ spanning arborescences $T_{1}[V(T_{k+1}^{c})], \ldots, T_{k}[V(T_{k+1}^{c}], T_{k+1}^{c}$ in $D[V(T_{k+1}^{c})]$. By Claim~\ref{contraction}, setting $W:=V(T_{k+1}^{c})$ and noting that $A^{-}(W)=\emptyset$, there exists such a packing as Theorem~\ref{directed-version} demands.

\textbf{Case 2.2} $A[\cup_{i=1}^{c-1}V(T_{k+1}^{i}), V(T_{k+1}^{c})] \neq \emptyset$.

Let $F_{2}$ be a spanning branching with arc set $A(F_{2}):=A(T_{k+1}^{c})$ and root set $R(F_{2})=\{ r\} \cup (\cup_{i=1}^{c-1}V(T_{k+1}^{i}))$. Since there exist $k$ spanning arborescences $T_{1}, \ldots, T_{k}$ and a spanning branching $F_{2}$ in $D- \cup_{i=1}^{c-1}A(T_{k+1}^{i})$, by Lemma~\ref{k+extra}, setting $U:= \cup_{i=1}^{c-1}V(T_{k+1}^{i})$, we have the following observation.

\begin{observation}
For any $\mathcal{P} \in \mathcal{D}(V)$, we have
\begin{equation}\label{k-spanning-arborescences}
\sum_{X \in \mathcal{P}}d_{A \setminus \cup_{i=1}^{c-1}A(T_{k+1}^{i})}^{-}(X) \geq k(|\mathcal{P}|-1),
\end{equation}
\begin{equation}\label{extra-branching}
\sum_{X \in \mathcal{P}}d_{A \setminus \cup_{i=1}^{c-1}A(T_{k+1}^{i})}^{-}(X) \geq k(|\mathcal{P}|-1)+\sum_{X \in \mathcal{P}}|P_{k+1}(X)|-1,
\end{equation}
where $P_{k+1}(X):=\{k+1: X \cap U=\emptyset \}$.
\end{observation}

Denote $\mathcal{D}_{1}(V):=\{\mathcal{P} \in \mathcal{D}(V): \mathcal{P}~ \text{witnesses  the equality of (\ref{k-spanning-arborescences})} \}$, and $\mathcal{D}_{2}(V)$ $:=$ $\{\mathcal{P}$ $\in \mathcal{D}(V): \mathcal{P}~\text{witnesses the equality of (\ref{extra-branching})} \}$.

\begin{claim}\label{equality-condition}
(i)  For any $\mathcal{P} \in \mathcal{D}_{1}(V) $, there exists at most one  $X \in \mathcal{P} $ such that $X \cap V(T_{k+1}^{c}) \neq \emptyset$.\\
(ii)  For any $X \in \mathcal{P} \in \mathcal{D}_{2}(V)$, $X \subseteq U$ or $X \subseteq V(T_{k+1}^{c})$.
\end{claim}
\begin{proof}
Let $ \mathcal{P} \in \mathcal{D}(V)$.  For any $X \in \mathcal{P}$ and $1 \leq i \leq k$, if $r_{i} \notin X$, then there exists a directed path in $T_{i}$ from $r_{i}$ to $X$; thus $d_{T_{i}}^{-}(X) \geq 1$. Since $r_{i}$ belongs to at most one $X \in \mathcal{P}$, we have $\sum_{X \in \mathcal{P}}d_{T_{i}}^{-}(X) \geq |\mathcal{P}|-1$ and
\begin{equation}\label{arborescence-part}
\sum_{X \in \mathcal{P}}d_{\cup_{i=1}^{k}A(T_{i})}^{-}(X) \geq k(|\mathcal{P}|-1).
\end{equation}

(i) Suppose $\mathcal{P} \in \mathcal{D}_{1}(V)$, and suppose to the contrary that there exist two subsets $X_{1}, X_{2} \in \mathcal{P} $ such that $X_{i} \cap V(T_{k+1}^{c}) \neq \emptyset$ for $i=1, 2$. Since $X_{1} \cap X_{2} = \emptyset$, without loss of generality, suppose $r \notin X_{2}$. Then there exists a directed path in $T_{k+1}^{c}$ from $r$ to $X_{2}$; thus $d^{-}_{T_{k+1}^{c}}(X_{2}) \geq 1$. Since $(\cup_{i=1}^{k}A(T_{i})) \cup A(T_{k+1}^{c}) \subseteq A\setminus \cup_{i=1}^{c-1} A(T_{k+1}^{i})$, we have
\[
\begin{split}
\sum_{X \in \mathcal{P}}d_{A \setminus \cup_{i=1}^{c-1}A(T_{k+1}^{i})}^{-}(X)
& \geq \sum_{X \in \mathcal{P}}d_{\cup_{i=1}^{k}A(T_{i})}^{-}(X)+d^{-}_{T_{k+1}^{c}}(X_{2}) \\
& \geq k(|\mathcal{P}|-1)+ 1 ~~\text{(by (\ref{arborescence-part}) and since $d^{-}_{T_{k+1}^{c}}(X_{2}) \geq 1$)},
\end{split}
\]
contradicting that $\mathcal{P} \in \mathcal{D}_{1}(V)$.

(ii) Suppose $\mathcal{P} \in \mathcal{D}_{2}(V)$. Note that for $X \in \mathcal{P}$,
\[
|P_{k+1}(X)|=
\begin{cases}
1, \text{if $X\cap U = \emptyset$, that is, $X \subseteq V(T_{k+1}^{c})$,} \\
0, \text{else.}\\
\end{cases}
\]
Without loss of generality, suppose that $ \{ X \in \mathcal{P}: X \subseteq V(T_{k+1}^{c}) \}$ consists of $X_{1}, \ldots, X_{t}$. Suppose to the contrary that there exists some $X_{0} \in \mathcal{P} $ such that $X_{0} \cap V(T_{k+1}^{c}) $ and $X_{0} \cap U$ are both nonempty.  For any $0 \leq i \leq t$, if $r \notin X_{i}$, then there exists a directed path in $T_{k+1}^{c}$ from $r$ to $X_{i}$; thus $d^{-}_{T_{k+1}^{c}}(X_{i}) \geq 1$. Since $r$ belongs to at most one of $X_{0}, \ldots, X_{t}$, we have
\begin{equation}\label{branching-part}
 \sum_{i=0}^{t}d^{-}_{T_{k+1}^{c}}(X_{i}) \geq t= \sum_{X \in \mathcal{P}} |P_{k+1}(X)|.
\end{equation}
Hence,
\[
\begin{split}
\sum_{X \in \mathcal{P}}d_{A \setminus \cup_{i=1}^{c-1}A(T_{k+1}^{i})}^{-}(X)
& \geq \sum_{X \in \mathcal{P}}(d_{\cup_{i=1}^{k}A(T_{i})}^{-}(X)+d^{-}_{T_{k+1}^{c}}(X))\\
& \geq k(|\mathcal{P}|-1)+ \sum_{X \in \mathcal{P}} |P_{k+1}(X)|~~\text{(by (\ref{arborescence-part}) and (\ref{branching-part}))},
\end{split}
\]
contradicting that $\mathcal{P} \in \mathcal{D}_{2}(V)$.
\end{proof}

\textbf{Case 2.2.1} There exists some $\overrightarrow{u_{0}v_{0}} \in A[ U, V(T_{k+1}^{c})]$ such that for any
 $X \in \mathcal{P} \in \mathcal{D}_{1}(V)$, $\overrightarrow{u_{0}v_{0}}$ doesn't enter $X$.

Suppose $u_{0} \in V(T_{k+1}^{i_{1}})$ for some $1 \leq i_{1} \leq c-1$. Note that $T_{k+1}^{i_{1}} + \overrightarrow{u_{0}v_{0}}$ is still an arborescence.

\begin{claim}\label{update}
Set $T_{k+1}^{i_{1}}:=T_{k+1}^{i_{1}} + \overrightarrow{u_{0}v_{0}} $ and $U:= U+v_{0}$. For any $\mathcal{P} \in \mathcal{D}(V)$, the inequalities (\ref{k-spanning-arborescences}) and (\ref{extra-branching}) still hold.
\end{claim}
\begin{proof}
Before the update of $ T_{k+1}^{i_{1}}$, note that
\begin{equation}\label{nonnegative-1}
\sum_{X \in \mathcal{P}}d_{ A \setminus   \cup_{i=1}^{c-1}A(T_{k+1}^{i})}^{-}(X)-k(|\mathcal{P}|-1) \geq
\begin{cases}
0, \text{if $\mathcal{P} \in \mathcal{D}_{1}(V)$},\\
1, \text{if $\mathcal{P} \notin \mathcal{D}_{1}(V)$.}
\end{cases}
\end{equation}
After the update, for any $\mathcal{P} \in \mathcal{D}_{1}(V)$, since $\overrightarrow{u_{0}v_{0}}$ doesn't enter any subset in $\mathcal{P}$ under \textbf{Case 2.2.1}, $\sum_{X \in \mathcal{P}}d_{ A \setminus   \cup_{i=1}^{c-1}A(T_{k+1}^{i})}^{-}(X)$ doesn't change;  for any $\mathcal{P} \notin \mathcal{D}_{1}(V)$,  $\sum_{X \in \mathcal{P}}d_{ A \setminus   \cup_{i=1}^{c-1}A(T_{k+1}^{i})}^{-}(X) $ decreases by at most $1$. So the left side of (\ref{nonnegative-1}) is nonnegative, that is, (\ref{k-spanning-arborescences}) still holds.

Before the update of the $T_{k+1}^{i_{1}}$ and $U$, note that
\begin{equation}\label{nonnegative-2}
\sum_{X \in \mathcal{P}}d_{ A \setminus   \cup_{i=1}^{c-1}A(T_{k+1}^{i})}^{-}(X)-k(|\mathcal{P}|-1)-(\sum_{X \in \mathcal{P}}|P_{k+1}(X)|-1) \geq
\begin{cases}
0, \text{if $\mathcal{P} \in \mathcal{D}_{2}(V)$},\\
1, \text{if $\mathcal{P} \notin \mathcal{D}_{2}(V)$.}
\end{cases}
\end{equation}

For any $\mathcal{P} \in \mathcal{D}_{2}(V)$, if $\overrightarrow{u_{0}v_{0}}$  enters some subset  $X_{1} \in \mathcal{P}$, then $v_{0} \in X_{1} \cap V(T_{k+1}^{c}) $;
by Claim~\ref{equality-condition} (ii),  we have $X_{1} \subseteq V(T_{k+1}^{c})$;
thus $|P_{k+1}(X_{1})|$ decreases by $1$ after the update of $U$, implying  $\sum_{X \in \mathcal{P}}|P_{k+1}(X)|-1$ decreases by $1$. Clearly, $\sum_{X \in \mathcal{P}}d_{ A \setminus   \cup_{i=1}^{c-1}A(T_{k+1}^{i})}^{-}(X)$ decreases by $1$ after the update of $T_{k+1}^{i_{1}}$.
If $\overrightarrow{u_{0}v_{0}}$  doesn't enter any subset  $X \in \mathcal{P}$, $\sum_{X \in \mathcal{P}}d_{ A \setminus   \cup_{i=1}^{c-1}A(T_{k+1}^{i})}^{-}(X)$ and $\sum_{X \in \mathcal{P}}|P_{k+1}(X)|-1$ doesn't change after the update of $T_{k+1}^{i_{1}}$ and $U$.

For any $\mathcal{P} \notin \mathcal{D}_{2}(V) $, after the update, $\sum_{X \in \mathcal{P}}d_{ A \setminus   \cup_{i=1}^{c-1}A(T_{k+1}^{i})}^{-}(X)$ decreases by at most $1$ and $\sum_{X \in \mathcal{P}}|P_{k+1}(X)|-1$ doesn't increase.

Combining all cases above, by (\ref{nonnegative-2}), we have (\ref{extra-branching}) still holds.
\end{proof}

Setting $T_{k+1}^{i_{1}}:=T_{k+1}^{i_{1}} + \overrightarrow{u_{0}v_{0}} $ and $U:= U+v_{0}$,  by Lemma~\ref{k+extra} and Claim~\ref{update}, in $D-\cup_{i=1}^{c-1}A(T_{k+1}^{i})$, there exist $k$ spanning arborescences and another spanning branching $F_{3}$ such that $R(F_{3}) \supseteq U$ and $|R(F_{3})|=|U|+1$. Suppose $R(F_{3}) \setminus U=\{r_{0} \}$.  Then combining $U=\cup_{i=1}^{c-1}V(T_{k+1}^{i})$, $(V, A(F_{3}) \cup (\cup_{i=1}^{c-1}A(T_{k+1}^{i})))$ is a spanning $c$-branching with $ \{r_{1}, \ldots, r_{c-1}, r_{0}\}$ as its root set. Denote by $T_{k+1}^{c,*}$ the component of $F_{3}$, whose root is  $r_{0}$. Then $V(T_{k+1}^{c,*}) \subseteq V \setminus U \subsetneq V(T_{k+1}^{c})$, contradicting the choice of $T_{k+1}^{c}$.

\textbf{Case 2.2.2} For any $a \in A[U, V(T_{k+1}^{c})]$, there exists some $\mathcal{P} \in \mathcal{D}_{1}(V)$ such that $a$ enters some  $X \in \mathcal{P}$.

Let $X_{0} \subseteq V$ be minimal such that there exists $\overrightarrow{u_{1}v_{1}} \in A[U, V(T_{k+1}^{c})]$ and $\mathcal{P}_{1} \in \mathcal{D}_{1}(V)$ such that $\overrightarrow{u_{1}v_{1}}$ enters $X_{0}$ and $X_{0} \in \mathcal{P}_{1}$.

\begin{claim}\label{locate-X0}
$X_{0} \subseteq V(T_{k+1}^{c})$.
\end{claim}
\begin{proof}
Suppose to the contrary that $X_{0} \nsubseteq V(T_{k+1}^{c})$, that is, $X_{0} \cap U \neq \emptyset$. Since $\overrightarrow{u_{1}v_{1}} \in A[U, V(T_{k+1}^{c})]$ enters $X_{0}$, we have $v_{1} \in X_{0} \cap V(T_{k+1}^{c}) \neq \emptyset$. Denote $X_{1}:=X_{0} \cap V(T_{k+1}^{c})$. If $\mathcal{P}_{1}-X_{0}+X_{1} \in \mathcal{D}_{1}(V)$, then $X_{1} \subsetneq X_{0}$ is such that $\overrightarrow{u_{1}v_{1}} $ enters $X_{1}$ and $X_{1} \in \mathcal{P}_{1}-X_{0}+X_{1}$, contradicting the minimality of $X_{0}$. So $\mathcal{P}_{1}-X_{0}+X_{1} \notin \mathcal{D}_{1}(V)$, that is,
$$\sum_{X \in \mathcal{P}_{1}-X_{0}+X_{1}}d_{A \setminus \cup_{i=1}^{c-1}A(T_{k+1}^{i})}^{-}(X)>k(|\mathcal{P}-X_{0}+X_{1}|-1)=k(|\mathcal{P}|-1). $$
Since $\mathcal{P}_{1} \in \mathcal{D}_{1}(V)$, we have
$$\sum_{X \in \mathcal{P}_{1}-X_{0}+X_{1}}d_{A \setminus \cup_{i=1}^{c-1}A(T_{k+1}^{i})}^{-}(X) >\sum_{X \in \mathcal{P}_{1}}d_{A \setminus \cup_{i=1}^{c-1}A(T_{k+1}^{i})}^{-}(X),$$
that is, $d_{A \setminus \cup_{i=1}^{c-1}A(T_{k+1}^{i})}^{-}(X_{1})> d_{A \setminus \cup_{i=1}^{c-1}A(T_{k+1}^{i})}^{-}(X_{0})$. Considering $ X_{1} =X_{0} \cap V(T_{k+1}^{c})$, there exists some $\overrightarrow{u_{2}v_{2}} \in A$ with $u_{2} \in X_{0} \setminus X_{1}$ and $v_{2} \in  X_{1}$.
Of course, $\overrightarrow{u_{2}v_{2}} \in A[U, V(T_{k+1}^{c})]$. By the hypothesis of \textbf{Case 2.2.2}, there exists $\mathcal{P}_{2} \in \mathcal{D}_{1}(V)$ such that $\overrightarrow{u_{2}v_{2}}$ enters some $X_{2} \in \mathcal{P}_{2}$.

Note that $v_{2} \in X_{0} \cap X_{2} \neq \emptyset $, $u_{2} \in X_{0} \setminus X_{2} \neq  \emptyset$; by the minimality of $X_{0}$, $X_{2} \setminus X_{0} \neq \emptyset $. So $X_{0}$ and $X_{2}$ are properly intersecting.  We adopt PIEO$^{1}$s in $\mathcal{G}_{0}=\mathcal{P}_{1} \uplus \mathcal{P}_{2}$, step by step, and obtain families   $\mathcal{G}_{0}, \mathcal{G}_{1}=\mathcal{G}_{0}-X_{0}-X_{2}+X_{0}\cap X_{2}+X_{0} \cup X_{2}, \ldots, \mathcal{G}_{n}$ of subsets  of $V$.
Recall that $\mathcal{G}_{i}'$ is the family of maximal elements in  $\mathcal{G}_{i}$, $\mathcal{P}_{3}:= \mathcal{G}_{n}' $ and $\mathcal{P}_{4}:=\mathcal{G}_{n} \setminus \mathcal{P}_{3} $. Since the indegree funtion $d_{A \setminus \cup_{i=1}^{c-1}A(T_{k+1}^{i})}^{-}(*)$ defined on $2^{V}$ is submodular, we have
\[
\begin{split}
k(|\mathcal{P}_{1}|-1)+k(|\mathcal{P}_{2}|-1)
& = \sum_{X \in \mathcal{P}_{1}} d_{A \setminus \cup_{i=1}^{c-1}A(T_{k+1}^{i})}^{-}(X)+\sum_{X \in \mathcal{P}_{2}} d_{A \setminus \cup_{i=1}^{c-1}A(T_{k+1}^{i})}^{-}(X)\\
& ~~~\text{(by $\mathcal{P}_{1}, \mathcal{P}_{2} \in \mathcal{D}_{1}(V)$)}
\end{split}
\]

\[
\begin{split}
~~~~~~~~~~~~~~~~~~~~~~~~~~~~~~~~
& =\sum_{X \in \mathcal{G}_{0}} d_{A \setminus \cup_{i=1}^{c-1}A(T_{k+1}^{i})}^{-}(X)\\
& \geq \sum_{X \in \mathcal{G}_{1}} d_{A \setminus \cup_{i=1}^{c-1}A(T_{k+1}^{i})}^{-}(X)~~~\text{(by submodularity)}\\
&  \ldots\\
& \geq \sum_{X \in \mathcal{G}_{n}} d_{A \setminus \cup_{i=1}^{c-1}A(T_{k+1}^{i})}^{-}(X)\\
& =\sum_{X \in \mathcal{P}_{3}} d_{A \setminus \cup_{i=1}^{c-1}A(T_{k+1}^{i})}^{-}(X)+\sum_{X \in \mathcal{P}_{4}} d_{A \setminus \cup_{i=1}^{c-1}A(T_{k+1}^{i})}^{-}(X)\\
& \geq  k(|\mathcal{P}_{3}|-1)+k(|\mathcal{P}_{4}|-1) ~~~\text{(by (\ref{k-spanning-arborescences}))}.
\end{split}
\]
Since $ |\mathcal{P}_{1}|+|\mathcal{P}_{2}|=|\mathcal{P}_{3}|+|\mathcal{P}_{4}|$, all ``$\geq$''s should be ``$=$''s. So we have $\mathcal{P}_{4} \in \mathcal{D}_{1}(V)$.

Since $X_{0} \cap X_{2} \subsetneq X_{0} \cup X_{2}$ in $\mathcal{G}_{1}$, we have $X_{0} \cap X_{2} \in \mathcal{G}_{1} \setminus \mathcal{G}_{1}'$. During the process of PIEO$^{1}$s, $X_{0} \cap X_{2} \in \mathcal{G}_{i} \setminus \mathcal{G}_{i}'$ for $1 \leq i \leq n$. So $X_{0} \cap X_{1} \in \mathcal{G}_{n} \setminus \mathcal{G}_{n}'=\mathcal{P}_{4}$, combing that $\overrightarrow{u_{2}v_{2}} \in A[U, V(T_{k+1}^{c})]$ enters $X_{0} \cap X_{2}$ and $\mathcal{P}_{4} \in \mathcal{D}_{1}(V)$, which contradicts the minimality of $X_{0}$.
\end{proof}

For $1 \leq i \leq k$ and $X \in \mathcal{P}_{1}$, if $r_{i} \notin X$, since $T_{i}$ is a spanning $r_{i}$-arborescence, there exists a directed path in $T_{i}$ from $r_{i}$ to $X$; thus we have $d_{T_{i}}^{-}(X) \geq 1$. So for $X \in \mathcal{P}_{1}$, we have
\begin{equation}\label{X-equation}
d_{\cup_{i=1}^{c}A(T_{i})}^{-}(X) \geq |\{i \in [k]: r_{i} \notin X \}|.
\end{equation}
Since $\cup_{i=1}^{k}A(T_{i}) \subseteq A \setminus \cup_{i=1}^{c-1}A(T_{k+1}^{i})$, we have
\begin{equation}\label{A-X-equation}
d_{A \setminus \cup_{i=1}^{c-1}A(T_{k+1}^{i})}^{-}(X) \geq d_{\cup_{i=1}^{c}A(T_{i})}^{-}(X).
\end{equation}
Hence, by (\ref{X-equation}) and (\ref{A-X-equation}), we have
\begin{equation}\label{proof-arborescence-part}
\begin{split}
\sum_{X \in \mathcal{P}_{1}}d_{A \setminus \cup_{i=1}^{c-1}A(T_{k+1}^{i})}^{-}(X)
& \geq \sum_{X \in \mathcal{P}_{1}} |\{i \in [k]: r_{i} \notin X \}|\\
& =\sum_{X \in \mathcal{P}_{1}} (k-|\{i \in [k]: r_{i} \in X \}|)
\end{split}
\end{equation}
\[
\begin{split}
~~~~~~~~~~~~~~~~~~
& =k|\mathcal{P}_{1}|- \sum_{X \in \mathcal{P}_{1}}|\{i \in [k]: r_{i} \in X \}|\\
& \geq k|\mathcal{P}_{1}|-k.
\end{split}
\]
Since $\mathcal{P}_{1} \in \mathcal{D}_{1}(V)$, we have the equalities of (\ref{proof-arborescence-part}) hold and thus the equalities of (\ref{X-equation}) and (\ref{A-X-equation}) hold.

Without loss of generality, suppose $r_{1}, \ldots, r_{s} \notin X_{0}$  and $r_{s+1}, \ldots, r_{k} \in X_{0}$.
The equality of (\ref{X-equation}) for $X=X_{0}$ implies that
\[
d_{T_{i}}^{-}(X_{0})=
\begin{cases}
1, 1 \leq i \leq s,\\
0, s+1 \leq i \leq k.
\end{cases}
\]

For $ 1 \leq i \leq s$, suppose $\overrightarrow{x_{i}y_{i}}$ is the arc in $T_{i}$ entering $X_{0}$. Considering $r_{i} \notin X_{0}$ and $d_{T_{i}}^{-}(X_{0})=1$, $T_{i}[X_{0}]$ is an arborescence  rooted at $y_{i} $ (otherwise, $d_{T_{i}}^{-}(X_{0}) \geq 2$, a contradiction).
For $s+1 \leq i \leq k$, considering $r_{i} \in X_{0}$ and $d_{T_{i}}^{-}(X_{0})=0$,  $T_{i}[X_{0}]$ is an arborescence  rooted at $r_{i}$ (otherwise, $d_{T_{i}}^{-}(X_{0}) \geq 1$, a contradiction).

Since $X_{0} \subseteq V(T_{k+1}^{c})$ by Claim~\ref{locate-X0}, we have $ d_{A \setminus \cup_{i=1}^{c-1}A(T_{k+1}^{i})}^{-}(X_{0})=d_{A}^{-}(X_{0})$.
The equality of (\ref{A-X-equation}) for $X=X_{0}$ implies that  $d_{A}^{-}(X_{0})= d_{\cup_{i=1}^{c}A(T_{i})}^{-}(X)$. So $A^{-}(X_{0})$ consists of $\overrightarrow{x_{1}y_{1}}$, \ldots, $\overrightarrow{x_{s}y_{s}}$, and $d_{T_{k+1}^{c}}^{-}(X_{0})=0$. Considering $X_{0} \subseteq V(T_{k+1}^{c})$,  $T_{k+1}^{c}[X_{0}]$ is an arborescence rooted at $r_{0}$.

First, suppose $|X_{0}| \geq 2$. To summarize the above, we have $|X_{0}| \geq 2$, $A^{-}(X_{0})$ consists of $\overrightarrow{x_{1}y_{1}}$, \ldots, $\overrightarrow{x_{s}y_{s}}$, and there exist $k+1$ arborescences $T_{1}[X_{0}], \ldots, T_{k}[X_{0}], T_{k+1}^{c}[X_{0}]$ in $D[X_{0}]$, the first $s$ of which are rooted at $y_{1}, \ldots, y_{s}$, respectively.  By Claim~\ref{contraction}, there exists such a packing as Theorem~\ref{directed-version} demands.

Finally, suppose $|X_{0}|=1$. The equalities of (\ref{proof-arborescence-part}) imply  $\sum_{X \in \mathcal{P}_{1}}|\{i \in [k]: r_{i} \in X \}|=k$, that is, $\{ r_{1}, \ldots, r_{k}\} \subseteq \cup \mathcal{P}_{1}$. Recall that $T_{k+1}^{c}[X_{0}]$ is an arborescence rooted at $r_{0}$; thus $r_{0} \in X_{0}$. Considering $|X_{0}|=1$, we have $X_{0}=\{ r_{0}\}$. By Claim~\ref{equality-condition} (i), for any $X \in \mathcal{P}_{1}- X_{0}$, $X  \cap V(T_{k+1}^{c})= \emptyset$; thus $\cup \mathcal{P}_{1} \cap V(T_{k+1}^{c})=X_{0}= \{ r_{0}\}$.  So $ \{r_{0}, r_{1}, \ldots, r_{k}\} \cap V(T_{k+1}^{c})= \{ r_{0}\}$.

For any $X \subseteq V(T_{k+1}^{c})-r_{0}$ and $ 1 \leq i \leq k$,  since there exists a directed path in $T_{i}$ or  $T_{k+1}^{c} $ from their root to $X$, respectively, we have $d_{T_{i}}^{-}(X) \geq 1$ and $ d_{T_{k+1}^{c}}^{-}(X) \geq 1$. Hence,  for $X \subseteq V(T_{k+1}^{c})-r_{0}$, we have
\begin{equation}\label{out-of-r0}
d_{A}^{-}(X) \geq k+1.
\end{equation}

By the hypothesis of \textbf{Case 2}, without loss of generality, let $|\mathcal{P}_{0}|$ be minimum such that $\mathcal{P}_{0}$ witnesses the equality of (\ref{frac-part}).

\begin{claim}\label{estimate-P0}
For any $X \in \mathcal{P}_{0}$, if $r_{0} \notin X$, then $X \cap U \neq \emptyset$.
\end{claim}
\begin{proof}
Suppose to the contrary that there exists some $X_{3} \in \mathcal{P}_{0} $ such that $X_{3} \subseteq V(T_{k+1}^{c})-r_{0}$.  By (\ref{out-of-r0}), we have $d_{A}^{-}(X_{3}) \geq k+1$; since $\mathcal{P}_{0}$ witnesses the equality of (\ref{frac-part}), we have
\[
\begin{split}
\sum_{X \in \mathcal{P}_{0}-X_{3}}d_{A}^{-}(X)
& =\sum_{X \in \mathcal{P}_{0}}d_{A}^{-}(X)-d_{A}^{-}(X_{3})\\
& \leq k(|\mathcal{P}_{0}|-1)+ |\mathcal{P}_{0}|-c-(k+1)\\
& =k(|\mathcal{P}_{0}-X_{3}|-1)+|\mathcal{P}_{0}-X_{3}|-c.
\end{split}
\]
By (\ref{frac-part}), we have $\sum_{X \in \mathcal{P}_{0}-X_{3}}d_{A}^{-}(X) \geq k(|\mathcal{P}_{0}-X_{3}|-1)+|\mathcal{P}_{0}-X_{3}|-c$. So $ \mathcal{P}_{0}-X_{3}$ witnesses the equality of (\ref{frac-part}), contradicting to the choice of $\mathcal{P}_{0}$.
\end{proof}

By Claim~\ref{estimate-P0}, each $X \in \mathcal{P}_{0}$ except at most one intersects $U$. So
\begin{equation}\label{estimate-union-T-i}
|U|=\sum_{i=1}^{c-1}|V(T_{k+1}^{i})| \geq |\mathcal{P}_{0}|-1.
\end{equation}
Since $\mathcal{P}_{0}$ witnesses the equality of (\ref{frac-part}) and $\nu_{f}(D) > k+ \frac{d-1}{d}$, we have
$$\sum_{X \in \mathcal{P}_{0}} d_{A}^{-}(X)=k(|\mathcal{P}_{0}|-1)+ |\mathcal{P}_{0}|-c>k(|\mathcal{P}_{0}|-1)+\frac{d-1}{d}(|\mathcal{P}_{0}|-1), $$
implying
\begin{equation}\label{estimate-P0-cd}
|\mathcal{P}_{0}|-1>(c-1)d.
\end{equation}
By (\ref{estimate-union-T-i}) and (\ref{estimate-P0-cd}), we have
\[
\frac{\sum_{i=1}^{c-1}|V(T_{k+1}^{i})|}{c-1}>d.
\]
So there exists some $1 \leq i_{2} \leq c-1$ such that $|V(T_{k+1}^{i_{2}})|>d$, that is, $T_{k+1}^{i_{2}}$ has at least $d$ arcs, completing the proof.

\noindent \textbf{Sharpness of the bound.} Let $k, d$ be integers with $k \geq d+1 \geq 2$. Let $G=(V, E)$ be an undirected graph such that $V=\{v_{1}, \ldots, v_{d+1} \}$, $|E|=kd+d-1$ and $\gamma_{f}(G)<k+1$.

By Theorem~\ref{Nash-forest-decomposition}, since $\gamma_{f}(G)<k+1$, $G$ can be decomposed into $k+1$ edge disjoint  spanning forests $F_{1}, \ldots, F_{k+1}$. Denote by $c_{i}$ the number of components of $F_{i}$ for $1 \leq i \leq k+1$. Then considering
$$kd+d-1=\sum_{i=1}^{k+1}|E(F_{i})|=\sum_{i=1}^{k+1}(d+1-c_{i}), $$
we have $\sum_{i=1}^{k+1}c_{i}=k+2$; thus all $c_{i}=1$ except exactly one $c_{i}=2$. Considering $k \geq d+1=|V|$, orient each $F_{i}$ to be a spanning branching such that $\cup_{i=1}^{k+1}R(F_{i})=V$.  Denote by $D=(V, A)$ the obtained digraph. Then for any $v \in V$, $d_{A}^{-}(v)=k+1-|\{i \in [k+1]:v \in R(F_{i}) \}|$.

Since $\gamma(G)<k+1$, we have for any $\emptyset \neq X \subseteq V $,
\begin{equation}\label{estimate-E(X)}
|E(X)| \leq (k+1)(|X|-1);
\end{equation}
the equality of (\ref{estimate-E(X)}) holds only when $|X|=1$.
For any $\mathcal{P} \in \mathcal{D}(V)$, we have
\begin{equation}\label{estimate-indegree-v}
\begin{split}
\sum_{v \in \cup \mathcal{P}}d_{A}^{-}(v)
& =\sum_{v \in \cup \mathcal{P}} (k+1-|\{i \in [k+1]:v \in R(F_{i}) \}|)\\
& =(k+1) |\cup \mathcal{P}|- \sum_{v \in \cup \mathcal{P}}|\{i \in [k+1]:v \in R(F_{i}) \}|\\
& = (k+1) |\cup \mathcal{P}|-\sum_{i=1}^{k+1}|R(F_{i}) \cap (\cup \mathcal{P})|\\
& \geq (k+1) |\cup \mathcal{P}|-\sum_{i=1}^{k+1}|R(F_{i})| \\
& = (k+1) |\cup \mathcal{P}|-(k+2);
\end{split}
\end{equation}
since $\cup_{i=1}^{k+1} R(F_{i})=V$, the equality of (\ref{estimate-indegree-v}) holds only when $\cup \mathcal{P}=V$.
So for any $\mathcal{P} \in \mathcal{D}(V)$, we have
\[
\begin{split}
\sum_{X \in \mathcal{P}} d_{A}^{-}(X)
& = \sum_{v \in \cup \mathcal{P}}d_{A}^{-}(v)-\sum_{X \in \mathcal{P}}|A(X)|\\
& \geq (k+1)|\cup \mathcal{P}|-(k+2)-\sum_{X \in \mathcal{P}}(k+1)(|X|-1) \text{(by (\ref{estimate-E(X)}) and (\ref{estimate-indegree-v}))}\\
& =(k+1)(|\mathcal{P}|-1)-1;
\end{split}
\]
the equality holds only when $\mathcal{P}=\mathcal{P}_{0}=\{\{v_{1}\}, \ldots, \{v_{d+1} \} \}$. So $\frac{\sum_{X \in \mathcal{P}}d_{A}^{-}(X)}{|\mathcal{P}|-1} \geq k+1$ if $\mathcal{P} \neq \mathcal{P}_{0}$; $\frac{\sum_{X \in \mathcal{P}_{0}}d_{A}^{-}(X)}{|\mathcal{P}_{0}|-1} = k+ \frac{d-1}{d}$. Hence, $\nu_{f}(D)=k+\frac{d-1}{d}$.

Since $ |V|=d+1$, an arborescence with at least $d$ arcs is spanning; Considering $|A| <(k+1)d$, $D$ doesn't contain such a packing as Theorem~\ref{directed-version} demands.

\section{Concluding remarks}

One can see that the conditions of Theorems~\ref{undirected-version} and \ref{directed-version}  are very similar. In most cases, the small difference is that the condition for packing trees in undirected graphs is about the lower bound of $E(G/\mathcal{P})$, where $\mathcal{P}$ is a partition, while the condition for packing arborescences in digraphs is about the lower bound of $\sum_{X \in \mathcal{P}} d_{A}^{-}(X)$, where $\mathcal{P}$ is a subpartition; readers can refer to \cite{HMS-25} for more examples.

As for packing mixed arborescences in mixed graphs, one can see a generalization of Theorem~\ref{Nash-Williams-Tutte} and \ref{packing-spanning-arborescence} in \cite{GY-21-spanning} by replacing $E(\mathcal{P})$ and $\sum_{X \in \mathcal{P}} d_{A}^{-}(X)$ with $E(\mathcal{P})+\sum_{X \in \mathcal{P}} d_{A}^{-}(X) $, where $\mathcal{P}$ is a subpartition. Hence, we give the following conjecture.

\begin{conjecture}\label{conjecture-new}
Let $H=(V, E, A)$ be a mixed graph, $k \geq 0$, $d \geq 1$ be integers.  If for any $\mathcal{P} \in \mathcal{D}(V)$, $\frac{E(\mathcal{P})+\sum_{X \in \mathcal{P}}d_{A}^{-}(X)}{|\mathcal{P}|-1} > k+ \frac{d-1}{d}$, then $H$ contains edge and arc disjoint $k$  spanning mixed arborescences and another mixed branching $F$ with
$|E(F)|+|A(F)|> \frac{d-1}{d}(|V|-1)$, and if $F$ is not a  spanning arborescence, then $F$ has a component with at least $d$ edges and arcs. Moreover, the bound is sharp.
\end{conjecture}

The key technic for proving Theorem~\ref{directed-version} is Lemma~\ref{pack-branching}, whose mixed graphic version, however, is unknown  (Hoppenot, Martin and Szigeti \cite{HMS-25} found a counterexample if we just replace $\sum_{X \in \mathcal{P}}d_{A}^{-}(X)$ in (\ref{condition-pack-branching}) with $|E(\mathcal{P})|+ \sum_{X \in \mathcal{P}}d_{A}^{-}(X)$). Thus, Conjecture~\ref{conjecture-new} may be false or true; but in the latter case, proving it is insufficient via our method and would necessitate a novel approach.

\acknowledgment{This work was supported by the National Natural Science Foundation of China (No. 12201623) and the Natural Science Foundation of Jiangsu Province (No. BK20221105)}

\section*{Declarations}

\noindent \textbf{Conflict of interest} There is no conflict of interest between the author or anyone else regarding this
manuscript.

\singlespacing

\end{document}